\documentclass{elsarticle}
\usepackage{amsfonts}

\usepackage{graphicx}
\usepackage{pstricks}
\usepackage{amsmath}
\usepackage{amsxtra}
\usepackage{amsfonts}
\usepackage{amssymb}%

\newtheorem{theorem}{Theorem}[section]

\newproof{proof}[theorem]{Proof}
\newtheorem{proposition}[theorem]{Proposition}
\newtheorem{corollary}[theorem]{Corollary}

\newtheorem{example}[theorem]{Example}
\newtheorem{remark}[theorem]{Remark}

\numberwithin{equation}{section}

\begin{document}

\begin{frontmatter}

\title{Berger measure for some transformations \\
of subnormal weighted shifts}
\author{Ra\' ul E. Curto\footnote{The first named author was partially supported by NSF Grants DMS-0801168 and DMS-1302666.}}
\address{Department of Mathematics, The University of Iowa, Iowa City, Iowa
52242}
\ead{raul-curto@uiowa.edu}

\author{George R. Exner}
\address{Department of Mathematics, Bucknell University, Lewisburg, Pennsylvania 17837}
\ead{exner@bucknell.edu}

\begin{abstract}
A subnormal weighted shift may be transformed to another shift in various ways, such as taking the $p$-th power of each weight or forming the Aluthge transform. \ We determine in a number of cases whether the resulting shift is subnormal, and, if it is, find a concrete representation of the associated Berger measure, directly for finitely atomic measures, and using both Laplace transform and Fourier transform methods for more complicated measures. \ Alternatively, the problem may be viewed in purely measure-theoretic terms as the attempt to solve  moment matching equations such as $(\int t^n \, d\mu(t))^2 = \int t^n \, d\nu(t)$ ($n=0, 1, \ldots$) for one measure given the other.
\end{abstract}


\end{frontmatter}

\section{Introduction and Preliminaries} \label{intro}

Let $\mathcal{H}$ be a separable complex infinite dimensional Hilbert space and $\mathcal{L}(\mathcal{H})$ be the algebra of bounded linear operators on $\mathcal{H}$. \ A subnormal weighted shift $W$ in $\mathcal{L}(\mathcal{H})$ (definitions reviewed below) is well known to have an associated Berger measure. \ Certain transformations of a subnormal weighted shift, yielding again a weighted shift, were studied in \cite{CPY} and \cite{Ex2} with a view to when the resulting shift is again subnormal. \ In the present paper we consider further these transformations, still with a view to whether the resulting shift is subnormal, but with the additional goal of determining concretely the associated Berger measure if it is. \ Of particular interest is the shift resulting from taking the $p$-th power of each weight in the cases $p=2$ (the Square Problem, in which the resulting shift is known to be subnormal) and $p = 1/2$ (the Square Root Problem, in which it may or may not be).

The Square Root Problem is related to the problem for the Aluthge transform of a shift. \ As noted above, the problem may be regarded instead as, given a measure, an attempt to find another measure satisfying certain moment matching equations. \ A motivating example is the Bergman shift (equivalently, the measure $ 1 \cdot \chi_{[0,1]} \, dt$);  it is known that every shift resulting from the $p$-th power of each weight ($p>0$), and the Aluthge transform (even the iterated Aluthge transforms), are subnormal, but the resulting Berger measures turn out to be somewhat surprising.

The organization of this paper as as follows. \ In the remainder of this section we introduce notation and give some background results. \ In the second section we give some preliminary results and consider the Square and Square Root Problems for finitely atomic measures, and in the third section we study some absolutely continuous measures. \ The fourth section uses Laplace transforms and solves in particular the $p$-th power problem for the Bergman shift, while in the fifth section we employ the Fourier transform. \ Finally, in Section \ref{section6} we conclude with some remarks and open questions.

A Hilbert space operator $T$ is normal if it commutes with its adjoint $T^*$, and is subnormal if it is the restriction of a normal operator to a (closed) invariant subspace. \ There has been considerable recent study of ``weak subnormalities,'' often using weighted shifts as test objects. \ Recall that an operator $T$ is $k$-hyponormal, $k = 1, 2, \ldots$, if $$\left( \begin{array}{ccccc}
    I & T^* & {T^*}^2 & \ldots & {T^*}^k  \\
    T & T^* T& {T^*}^2 T & \ldots & {T^*}^k T \\
    T^2 & T^* T^2& {T^*}^2 T^2 & \ldots & {T^*}^k T^2 \\
     \vdots &  & \vdots  &  & \vdots \\
     T^k & T^* T^k& {T^*}^2 T^k & \ldots & {T^*}^k T^k
                    \end{array} \right) \geq 0.$$
The well-known Bram-Halmos characterization of subnormality states that $T$ is subnormal if and only if it is $k$-hyponormal for all $k = 1, 2, \ldots$ (see
\cite{Br} and \cite{Em}). \ A somewhat different route to subnormality is via $n$-contractivity:  an operator $T$ is $n$-contractive ($n=1, 2, \ldots$) if
$$\sum_{j=0}^{n}(-1)^{j}\left( \begin{array}{c}
    n \\
    j
                    \end{array}\right){T^{\ast }}^{j}T^{j} \geq 0.$$
The Agler-Embry characterization of subnormality (under the harmless condition that the operator is a contraction, $\|T\| \leq 1$) states that $T$ is subnormal if and only if it is $n$-contractive for all $n$. \ (See \cite{Ag}, in which what is actually presented involves $n$-hypercontractivity, which difference is not of importance here.)

Let us recall some now familiar notation for weighted shift operators. \ Consider $\ell^2$ with the standard basis $\{e_j\}_{j=0}^\infty$ (note that indexing begins at zero). \ Given a
 weight sequence $\alpha: \alpha_0,\alpha_1, \alpha_2,
\ldots$ of positive numbers, define the weighted shift $W_\alpha$ on $\ell^2$ by $W_\alpha e_j := \alpha_j e_{j+1}$, extending
by linearity. \ The \underline{moments} of the shift are defined by
$\gamma_0 = 1$ and $\gamma_j =
\prod_{i=0}^{j-1} \alpha_i^2$, $j \geq 1$.
(The reader should note that some authors take the moments to be products of the weights, not of squares.)

Weighted shifts have proved a particularly tractable place to study the $k$-hyponormality and $n$-contractivity conditions (for finite $k$ and $n$) because the general conditions simplify considerably in the case of shifts. \ A weighted shift $W_\alpha$ is $k$-hyponormal if and only if the following  Hankel moment matrices are positive for $m = 0, 1, 2, \ldots$ :
$$\left( \begin{array}{ccccc}
    \gamma_m & \gamma_{m+1} & \gamma_{m+2} & \ldots & \gamma_{m+k}  \\
    \gamma_{m+1} & \gamma_{m+2} & & \ldots & \gamma_{m+k+1} \\
  \gamma_{m+2} & \ldots & & \ldots & \gamma_{m+k+2} \\
     \vdots &  & \vdots  &  & \vdots \\
     \gamma_{m+k} & \gamma_{m+k+1}&  & \ldots & \gamma_{m+2k}
                    \end{array} \right) \geq 0. $$
(Thus, an operator matrix condition is replaced by a scalar matrix condition \cite{Cu1}.) \
A similar (easier) simplification (\cite{Ex1}) is that the weighted shift is $n$-contractive iff $$\sum_{j=0}^{n}(-1)^{j}\binom{n}{j} \gamma_{m+j} \geq 0, \hspace{.2in}m =
0, 1, \ldots . $$

Recall that every subnormal weighted shift $W_\alpha$ has an associated Berger measure, that is, a probability measure $\mu$ supported on $[0, \|W_\alpha\|^2]$ and satisfying
$$\gamma_n = \int_0^{\|W_\alpha\|^2} t^n d \mu(t), \hspace{.2in} n
= 0, 1, \ldots .$$

We now digress briefly to make connection with the measure problem we consider. \ Recall that the Schur product of matrices is the entry-wise product, and define the Schur product of two shifts $W_\alpha$ and $W_\beta$ to be the shift $W_\alpha \circ W_\beta$ with weights $\alpha_n \beta_n$;  it is immediate to see that the moment sequence of the result is the Schur product of the two moment sequences. \ Since the Schur product of two positive matrices is positive (\cite{Pa}), we conclude that if $W_\alpha$ and $W_\beta$ are subnormal so is $W_\alpha \circ W_\beta$, because for any $k$,  each of the matrices involved in testing its $k$-hyponormality is the Schur product of the associated (positive) matrices for $k$-hyponormality of $W_\alpha$ and $W_\beta$. \ In particular, if $W_\alpha = W_\beta$ we have that for the Square Problem the resulting shift is subnormal, but note that no information is provided about its Berger measure even if the Berger measure of $W_\alpha$ is known.

If $W_\alpha$ is a weighted shift with weight sequence $\alpha: \alpha_0,\alpha_1, \alpha_2,
\ldots$, and $p >0$, we define the $p$-th power shift to be the shift with weight sequence $\alpha_0^p,\alpha_1^p, \alpha_2^p,
\ldots$. \ We will refer in particular to the square root shift of $W_\alpha$ ($p = 1/2$, denoted $W_{\sqrt{\alpha}}$) and the ``square'' ($p = 2$, denoted $W_{\alpha^2}$ or occasionally $W_\alpha^{(2)}$). \ Recall also that, for a general operator $T$ with $T \equiv U |T|$ being its polar decomposition (with ker $U=$ ker $T$),
the Aluthge transform is defined by $AT(T) := |T|^{\frac{1}{2}} U |T|^{\frac{1}{2}}$, and the iterated Aluthge transform $AT^n(\cdot)$ is given by
$AT^{n+1}(T) := AT(AT^n(T)) \; (n \ge 1)$. \ It is easy to compute that the Aluthge transform of a weighted shift is again a weighted shift. \ 

Further, there is a general relationship between the square root and Aluthge transform of a shift. \ Denote by $W_{r(\beta)}$ the restriction of $W_\beta$ to the canonical invariant subspace spanned by $\{e_1, e_2, \ldots\}$ (that is, the orthocomplement of the zeroth basis vector). \ One computes easily that $A(W_\alpha) = W_{\sqrt{\alpha}} \circ W_{r(\sqrt{\alpha})}$ and it follows as argued above that if the square root shift is subnormal (and hence certainly its restrictions are), then the Aluthge transform is subnormal as well.

Certain subnormal shifts, due to Agler  and used as models in (\cite{Ag}), have been studied (in particular, certain single weight perturbations have a tidy relationship between $k$-hyponormality and $n$-contractivity). \ The $j$-th Agler shift, $A_j \; (j = 2, 3, \ldots)$, has weight sequence $\alpha^{(j)}:
\sqrt{\frac{1}{j}},\sqrt{\frac{2}{j+1}},\sqrt{\frac{3}{j+2}},
\ldots$, and is known to have Berger measure $d \mu^{(j)}(t) = (j-1) (1-t)^{j-2} dt$. \ (Observe that $A_2$ is simply the familiar Bergman shift.) \ Transformations of these shifts under powers and Aluthge transforms turn out to be subnormal, via an approach based upon completely monotone functions. \ A function $f: \mathbb{R}_+ \rightarrow \mathbb{R}_+ \setminus \{0\}$ is {\em completely monotone} if its derivatives alternate in sign: $f^{(2 j)}$ is non-negative for $j \geq 1$, and $f^{(2 j + 1)}$ is non-positive for $j \geq 0$. \ It is routine to check that if some completely monotone function $f$ interpolates the moments of a shift in the sense that $f(m) = \gamma_m$ for all $m$, then the shift is subnormal (the complete monotonicity condition converts readily to positivity of the iterated forward differences of the moments as required for the $n$-contractivity conditions). \ It is by this route that the following result was obtained.

\begin{theorem} (\cite[Theorem 2.10]{Ex2}) \ 
For $j=2,3,\ldots$, let $A_j$ be the $j$-th Agler shift. \ Then \newline
(i) any $p$-th power transformation ($p > 0$) of $A_j$ is subnormal; and \newline
(ii) any $n$-th iterated Aluthge transform of $A_j$ is subnormal ($n=2,3.\ldots$).
\end{theorem}

We emphasize again that this mode of proof offers no information about the Berger measure of the resulting shift.

We may formulate our central question in two versions: operator-theoretic and measure-theoretic. \ Observe that there is no Square Problem in the operator-theoretic case, since we know that the Schur product of two subnormal operators is subnormal.

\medskip

\subsection{Operator-theoretic formulation of the Square Root Problem}

Given a subnormal weighted shift $W_\alpha$, under what conditions is its ``square root'' shift $W_{\sqrt{\alpha}}$ subnormal?

\medskip

Note also that by the uniqueness of solutions to the Hamburger moment problem, we may phrase the Square Root Problem as follows:

\medskip

\subsection{Measure-theoretic formulation of the Square Root Problem}

Consider the following ``moment matching'' equation:

\begin{equation}    \label{eq:momentmatching}
\int t^n \, d\mu(t) = \left(\int t^n \, d\nu(t) \right)^2 \hspace{.2in} (n = 0, 1, 2, \ldots)  .
\end{equation}

The \textit{Square Root Problem} can be stated as follows: Given a probability measure $\mu$ (supported on a compact interval in $\mathbb{R} _+$), does there exist a measure $\nu$ such that (\ref{eq:momentmatching}) holds?  \ If so, can one find $\nu$ in terms of $\mu$?

Observe also that in this latter formulation we have the \textit{Square Problem}: Given a probability measure $\nu$ (supported on a closed interval in $\mathbb{R} _+$), what is the measure $\mu$ so that (\ref{eq:momentmatching}) holds?

\medskip

There is a result concerning another sort of transformation of a weighted shift which we will have occasion to mention later, and is the earliest instance we are acquainted with which gives information about the Berger measure. \ It has been common in the study of subnormal shifts for weak subnormalities to consider a ``back step extension'': given a weight sequence $\alpha: \alpha_0,\alpha_1, \alpha_2, \ldots$, form the weight sequence $\alpha(x): x, \alpha_0,\alpha_1, \alpha_2, \ldots$ by prefixing a parameter and consider the resulting shift. \ The following is from \cite{CY} and is a model answer, in that it both answers the question of when the resulting shift is subnormal and provides the measure concretely if it is. \ Let $\delta_z$ denote the Dirac point mass at $z$.

\begin{theorem} \label{th:backstepmeas} (\cite[Proposition 1.5]{CY}) \  
Suppose $W_\alpha$ is a subnormal weighted shift with Berger measure $\mu$. \ Then there exists a subnormal weighted shift with weight sequence $x, \alpha_0, \alpha_1, \ldots$ (a subnormal ``back step extension'') if and only if

\medskip

\noindent (i) $\frac{1}{t} \in L^1(\mu)$, and \newline
\noindent (ii) $x^2 \leq (\|\frac{1}{t}\|_{L^1(\mu)|})^{-1}$.

\medskip

In this case, the Berger measure for the back step extension is
$$\frac{x^2}{t}d\mu(t) + (1-x^2 \|\frac{1}{t}\|_{L^1(\mu)|})d\delta_0(t).$$
\end{theorem}

\section{A Support Result and Atomic Measures} \label{section2}

We begin by recasting the moment matching equation (\ref{eq:momentmatching}) in terms of product measures. \ Clearly, for each $n$,
\begin{eqnarray*}
\int t^n \, d\mu(t) &=& \left(\int t^n \, d\nu(t) \right)^2\\
                    &=&\left(\int s^n \, d\nu(s) \right)\cdot \left(\int t^n \, d\nu(t) \right) \\
                    &=& \int \hspace{-.1in}\int s^n t^n \, d\nu(s)  d\nu(t)\\
                    &=& \int \hspace{-.1in}\int s^n t^n \, d(\nu\times\nu)(s,t).\\
                    \end{eqnarray*}
If we define $p:\mathbb{R}\times \mathbb{R} \rightarrow \mathbb{R}$ by $p(x,y) := xy$, the equation above suggests that $\mu = (\nu \times \nu)\circ p^{-1}$, which  is indeed correct. \ This was obtained independently, and in greater generality, in \cite[Lemma 3.1]{SS}, as was a crucial relationship about supports (\cite[Theorem 3.3]{SS}) which we record in the form needed for our investigations. \ (The proof of the second assertion may be done directly from measure theory (as in \cite{SS}) or less naturally if more neatly by using the spectral mapping theory and the relationship between support and spectrum for multiplication operators.) \ Let $\mbox{\rm supp}(\mu)$ denote the (closed) support of a measure $\mu$, and denote the closure of a set $S$ by $\overline{S}$.

\begin{theorem} (\cite[Lemma 3.1 and Theorem 3.3]{SS}) \ 
Suppose that $\mu$ and $\nu$ are measures supported in some compact subset of $\mathbb{R}_+$ satisfying (\ref{eq:momentmatching}). \ Then with $p(x,y) \equiv xy$ as above,
\begin{itemize}
\item[(i)]$\mu = (\nu \times \nu)\circ p^{-1}$ \\
\item[(ii)]  $\mbox{\rm supp}(\mu) = \overline{(\mbox{\rm supp}(\nu))^2}$.\\
\end{itemize}
\end{theorem}

For convenience of notation, we abbreviate the situation in which (\ref{eq:momentmatching}) and the relationships (i) and (ii) above hold by writing $\mu = \nu^2$. \ Observe that in this situation
$$\mu(E) = (\nu \times \nu)(p^{-1}(E)) = (\nu \times \nu)(\{(s,t): s t \in E\}).$$
It will be useful later to note a simple geometrical fact, that the $\mu$ measure of some  interval $[a,b]$ is the $\nu \times \nu$ measure of the planar region with hyperbolic boundaries $x y = a$ and $x y = b$.

We pause to record one easy result about either restrictions of shifts, or, if one prefers, measures of the form $t^n d\mu(t)$. \ Recall that given a weighted shift $W$ with Berger measure $\mu$, the restriction of $W$ to the canonical invariant subspace obtained from $\{e_0, e_1, \ldots, e_{n-1}\}^\perp$ has Berger measure $t^n \mu(t)$ up to the normalization factor $\int_0^1 t^n d\mu(t)$ (to obtain a probability measure). \ The following is immediate by comparison of moments.

\begin{proposition}  \label{prop:restfollow}
Suppose one has solved $\mu = \nu^2$ and fix $n \ge 1$;  then $t^n \mu \approx (t^n \nu)^2$ (up to normalization). \ More precisely,
$$\left(\frac{1}{\int_0^1 t^n d\mu(t)}\right) t^n d\mu(t) = \left(\left(\frac{1}{\int_0^1 t^n d\mu(t)}\right)^{1/2} t^n d \nu(t) \right)^2.$$
\end{proposition}

Turning to atomic measures, the Square Problem is easy. \ Recall that $\delta_z$ denotes the Dirac point mass at $z$, and, for a set $S$, let $S^2 := \{s \cdot t: s, t \in S\}$. \ The following comes simply from comparing moments.

\begin{proposition}  \label{prop:squareptFA}
Suppose $\nu \equiv \sum_{i=1}^\infty \phi_i \delta_{x_i}$ is a probability measure. \ Let $X \equiv \{x_i\}^{\infty}_{i=1}$. \ Then
$$\nu^2 = \sum_{z \in X^2} (\sum_{x_i\cdot x_j = z} \phi_i \phi_j) \delta_z.$$
\end{proposition}

We are interested in contraction operators and their Berger measures, and therefore, we henceforth specialize to the case in which $\|W_\alpha\| \leq 1$; thus, our measures are probability measures supported in $[0,1]$ and, for convenience, assume that $1$ is in the support. (This effectively means that $\|W_\alpha\| = 1$.) \ Also, when considering atomic measures, we assume as well that there is an atom at the point $1$.

For finitely atomic measures, there is a solution to the Square Root Problem -- in fact, two -- although they are not completely satisfactory.

\begin{proposition} Let $\mu$ be a finitely atomic probability measure with support contained in $[0,1]$, and having positive mass at an atom at $1$, say $\mu \equiv \sum_{i=0}^N \rho_i \delta_{x_i}$, where the $x_i$ are in increasing order and $\rho_i > 0$ for all $i$. \ (Note $x_N = 1$.) \ If there exists $\nu$ such that $\mu = \nu^2$, then
$$\mbox{\rm supp}(\nu)=\left\{
\begin{array}{rr}
\{0\} \cup ([\sqrt{x_1}, 1]\cap \mbox{\rm supp}(\mu))   & \mbox{\rm if }x_0 = 0 , \\
([\sqrt{x_0}, 1]\cap \mbox{\rm supp}(\mu))  & \mbox{\rm if }x_0 \neq 0 . \\
\end{array}%
\right.$$

\noindent If the square of the appropriate set above is not $\mbox{\rm supp}(\mu)$, then $\mu$ has no square root.

Further, given the measure $\mu$, one may solve ``algorithmically'' for the only possible ``candidate'' $\hat{\nu}$, by assigning the needed masses on the support of $\hat{\nu}$ one by one, in decreasing order of the point in the support. \ (The attempt may fail in various ways at some step, if no (positive) assignment is possible.) \ Alternatively, with $\{\gamma_n\}$ the moments for $\mu$, and $\hat{S} = \{s_0, s_1, \ldots, s_{M-1}\}$ the required support set above, assign masses $\varphi_k$ to $\hat{\nu}$ by solving a Vandermonde type equation

\medskip

\begin{equation}  \label{eq:vdM}
\left(
\begin{array}{cccc}
1 & 1 & \ldots & 1 \\
s_0 & s_1 & \ldots & s_{M-1} \\
s_0^2 & s_1^2 & \ldots & s_{M-1}^2 \\
\cdots & \cdots &  & \cdots \\
s_0^{M-1} & s_1^{M-1} & \ldots & s_{M-1}^{M-1} \\
\end{array}
\right) \cdot
\left(
\begin{array}{c}
\varphi_0  \\
\varphi_1   \\
\varphi_2   \\
\vdots  \\
\varphi_{M-1} \\
\end{array}
\right) =
\left(
\begin{array}{c}
\sqrt{\gamma_0} \\
\sqrt{\gamma_1}   \\
\sqrt{\gamma_2} \\
\vdots  \\
\sqrt{\gamma_{M-1}} \\
\end{array}
\right).
\end{equation}

This candidate need not be successful, and all but one of the masses of $\hat{\nu}^2$ must be compared to those of $\mu$ to determine if this candidate is satisfactory. \ Alternatively, we must ``match'' (at least) $N + 1$ (the cardinality of the support of $\mu$) of the square roots of the moments of $\mu$ (equivalently, extend the equation in (\ref{eq:vdM}) to have the obvious $N + 1$ rows).

\end{proposition}

\begin{proof}
{\bf Algorithmic approach (matching of masses)} \ Since we have assumed there is an atom of $\mu$ at $1$, it is clear from Proposition \ref{prop:squareptFA} that any possible solution $\nu$ must be supported in $[0,1]$ and must have an atom at $1$ as well. \ But then each atom of $\nu$ is also an atom of $\nu^2$, and the claims about the support of $\nu$ follow easily.

We assume henceforth that $M \geq 2$, and, if $M = 2$, that neither of the two atoms of $\mu$  is at zero. \ (The other results are trivial or are obtained by easy modifications of the arguments to follow.) \ Recall that $\mbox{\rm supp}(\mu) = \{x_0,\ldots,x_N\}$. \ By the support claim above, the support of any possible candidate $\hat{\nu}$ is of the form $\hat{S} = \{s_0, s_1, \ldots, s_{M-1}\}$, where $s_{M-1} = 1$ and either $s_0 = 0$ and $s_1 = \sqrt{x_1}$ (if $0 = x_0 \in \mbox{\rm supp}(\mu)$) or $s_0 = \sqrt{x_0}$ if $0 \notin \mbox{\rm supp}(\mu)$. \ To produce the only possible candidate algorithmically, observe that
$\hat{\nu}= \sum_{j=0}^{M-1} \varphi_j \delta_{s_j}$
must have $\varphi_{M-1}^2 = \rho_N$ (in order that $\hat{\nu}^2(\{1\}) = \mu(\{1\})$. \ But then, taking into account the equation in Prop.\ref{prop:squareptFA}, $\varphi_{M-2}$ (the mass associated with $s_{M-2} = x_{N-1}$) must be chosen to satisfy
$$2 \varphi_{M-1} \varphi_{M-2} = \rho_{N-1},$$
yielding
$$\varphi_{M-2} = \frac{\rho_{N-1}}{2 \sqrt{\rho_N}}.$$
Consider now the task of fixing a mass at $s_{M-3} = x_{N-2}$. \ If $x_{N-1}^2 > x_{N-2}$ then the process has failed and $\mu$ has no square root, because the square of the required support set for any possible $\hat{\nu}$ is not the support set for $\mu$. \ If $x_{N-1}^2 = x_{N-2}$, then $\varphi_{M-3}$ must be chosen in accordance with
$$2 \varphi_{M-3} \varphi_{M-1}  + \varphi_{M-2}^2 = \rho_{N-2};$$
if $x_{N-1}^2 < x_{N-2}$, then $\varphi_{M-3}$ must be chosen in accordance with
$$2 \varphi_{M-3} \varphi_{M-1} = \rho_{N-2}.$$
(Of course, if the relevant equation requires a value of $\varphi_{M-3}$ less than zero the process fails and there is no successful candidate, since we must be building a (positive) measure.) \ Continuing in this fashion, we will eventually either fail or have made (forced) assignments to all the $\varphi_j$ and produced the only possible candidate for $\nu$. \ (It is left to the reader to resolve the minor anomaly present if $s_0 = x_0 = 0$.)

Alternatively, to generate the candidate, it is well known (see \cite[Theorem 3.9]{CF}) that given a known finite support set and a known target set of moments, the masses $\varphi_j$ at the atoms $s_j$ may be obtained by solving the Vandermonde equation in the statement, yielding all the $\varphi_j$ ``at once.''

Either of these candidates may fail, even if the algorithm does not obviously fail in its execution; if what results is not a probability measure the candidate is not successful (note that the Vandermonde forces a probability measure). \ Further, it is easy to construct an example (by squaring a finitely atomic probability measure, perturbing the masses of the two smallest atoms slightly, and using the result as $\mu$) where an apparently viable candidate still fails. \ However, if an apparently viable \textit{probability} measure candidate when squared matches all but one of the probability measure target masses, it is successful since it then must match all of them.

{\bf Vandermonde approach (matching of moments)}. \ Alternatively, if a candidate has moments the appropriate values
$$\sqrt{\gamma_0}, \ldots, \sqrt{\gamma_{M-1}}, \sqrt{\gamma_M}, \ldots, \sqrt{\gamma_N}$$
 (which may be thought of as a solution to a non-square Vandermonde type equation which includes the equation yielding the candidate), the candidate is successful. \ This is certainly sufficient, since it ensures that $\hat{\nu}^2$ is a measure with support of size (at most) $N + 1$ and a subset of the support of $\mu$ matching $N + 1$ moments of $\mu$;  by the version of the Vandermonde equation for $\mu$, and invertibility of the Vandermonde matrix, this must be $\mu$. \ On the other hand, by a small perturbation of $\gamma_N$ to $\gamma_N^\prime$, we may clearly construct a measure $\mu^\prime$ with support set the same as that of $\mu$ and such that $\gamma_j = \gamma_j^\prime$ ($j = 0, \ldots, N-1$) but $\gamma_N \neq \gamma_N^\prime$; $\hat{\nu}^2$ cannot be both $\mu$ and $\mu^\prime$, so it is not sufficient to match fewer than $N + 1$ moments.
 \end{proof}
 
If one is willing to accept a transfinite algorithm, the algorithmic approach works for certain infinite support sets (such as $\{1, r, r^2, r^3, \ldots\}$). \ What is needed is that there is a finite collection of mass matching that can take place ``first,'' a finite number that can thereafter be solved ``second,'' and so on. \ However, we do not know how to handle some support set like
$$\{r^{p_n/2^n}: n \in \mathbb{N}, p_n \in \mathbb{Z}_+, \rho_n \leq 2^n \}.$$

The following corollary of the remarks above about support sets was also obtained in \cite{SS} and follows from a simple calculation.

\begin{corollary}
 Neither a two-atomic nor a four-atomic measure can have a square root, unless one atom is at zero.
 \end{corollary}

\begin{remark} \ (Measures with a geometric series as support). \ For some measures whose support is $\{r^n:n \ge 0\}$ for some $0<r<1$, we can obtain some results by re-interpreting the questions as one of generating functions. \ If
$$\mu = \sum_{n=0}^\infty \rho_n \delta_{r^n} \hspace{.2in} \mbox{\rm and }\hspace{.2in}
\nu = \sum_{n=0}^\infty \varphi_n \delta_{r^n},$$
to say $\mu = \nu^2$ is equivalent to
$$\sum_{n=0}^\infty \rho_n x^n = \left(\sum_{n=0}^\infty \varphi_n x^n \right)^2.$$
One solves by coefficient matching, equivalent to our ``algorithmic approach'' of matching masses. \ Since $r$ does not appear, it is clear that such an equation is solvable for one $r$ if and only if it is solvable for all $r$.

Sometimes the generating function approach gives an immediate answer. \ If $\rho_n = \left(\frac{1}{2}\right)^{n+1}$, then the generating function is
$$\frac{1}{2} + \frac{1}{2^2}z + \frac{1}{2^3} z^2 + \ldots = \frac{1}{2-z}.$$
But the expansion for its square root is
$$\frac{1}{\sqrt{2-z}} = \frac{\sqrt{2}}{2} + \frac{1}{8}z + \frac{3}{64}z^2 + \ldots$$
and the coefficients give the masses for the square root measure.
In general probability distributions on the non-negative integers may be re-interpreted on the set $\{r^n:n=0, 1, 2, \ldots\}$ and generating functions give square roots in some known cases. \ As well, if we choose $r$ and $s$ so that $r^n s^m = r^i s^j$ forces $n=i$ and $m=j$ ($r = 1/3$ and $s = 1/7$, say), products of generating functions in two variables may clearly be used in certain tractable cases.
\end{remark}

\section{Absolutely Continuous Measures} \label{section3}

For ease of presentation in what follows, we regard functions as living only on $[0,1]$, and omit products with $\chi_{[0,1]}$ throughout.

If we seek to solve the Square Problem $\mu = \nu^2$, and $\nu$ is absolutely continuous with respect to Lebesgue measure on $[0,1]$, writing $d\nu(t) = g(t) dt$ with $g$ the Radon-Nikodym derivative in $L^1$, it is reasonable to hope that $\mu$ is likewise absolutely continuous and to pursue its Radon-Nikodym derivative.

Recall that with $p(x,y) \equiv xy$, we have $\mu(E) = (\nu \times \nu)(p^{-1}(E))$. \ Observe first that $\mu$ is absolutely continuous with respect to Lebesgue measure; for, consider a Lebesgue null set $N \subseteq [0,1]$ covered by a union of (relatively) open intervals of small total measure. \ If $I$ is such a small open interval, $p^{-1}(I)$ is a narrow ``hyperbolic slice'' in the unit square;  the largest of these, in the case $I$ is an interval $[0, b)$, has area $b + 2b(-ln b)$. \ Therefore the $\nu \times \nu$ measure of this set is small, and thus the inverse image under $p$ of a null set is a $\nu \times \nu$ null set, and $\mu$ is absolutely continuous with respect to Lebesgue measure. \ Denote the Radon-Nikodym derivative of $\mu$ by $f$.

To find $f$, it suffices to find $f(a)$ for $0 < a < 1$, and to do this it is sufficient (a.e.) to take $$\lim_{n \rightarrow \infty} \frac{\mu([a, a+1/n])}{1/n}$$
(this approach is due to R. de Possel -- see \cite[page 269]{Ra}). \
Taking $E = [a, a + 1/n]$ with $0 < a < 1$ and $1/n$ small, and recalling that $p^{-1}(E)$ is the region of the unit square bounded by the hyperbolas $ x y = a$ and $ x y = a + 1/n$, we obtain
\begin{eqnarray*}
\mu([a, a + 1/n]) &=& \int \int_{p^{-1}(E)}  1 \, d\nu \times d\nu \\
                       &=& \int_a^{a + 1/n} \int_{a/x}^1 1 \, \, g(y) g(x) \, dy \, dx +  \\
                     & &  \hspace*{.1in} + \int_{a + 1/n}^1 \int_{a/x}^{(a + 1/n)/x} 1 \, \, g(y) g(x) \, dy \, dx  \\
                     &=&  \int_{a}^1 \int_{a/x}^{(a + 1/n)/x} 1 \, \, g(y) g(x) \, dy \, dx , \\
\end{eqnarray*}
(where we have used that $g$ vanishes outside $[0,1]$). \ Then with $f$ the Radon-Nikodym derivative of $\mu$,
\begin{eqnarray}
f(a) = \lim_{n \rightarrow \infty} \frac{\mu([a, a+1/n])}{1/n} &=& \lim_{n \rightarrow \infty}  \left(\frac{1}{1/n}\right) \int_{a}^1 \int_{a/x}^{(a + 1/n)/x} 1 \, \, g(y) g(x) \, dy \, dx  \nonumber \\
     &=&  \int_{a}^1 \left(\lim_{n \rightarrow \infty} \frac{\int_{a/x}^{(a + 1/n)/x} 1 \, \, g(y) \, dy }  {\frac{1}{n x}}\right) \, \frac{1}{x} g(x) \, dx  \nonumber \\
     &=&  \int_{a}^1 g(\frac{a}{x}) g(x) \frac{1}{x} \, dx  \hspace*{.2in} (\mbox{\rm a.e.}), \nonumber \\  \label{neweq} 
\end{eqnarray}
where the limit is moved inside the integral using (for example) the Lebesgue
Dominated Convergence Theorem since $\int g(y) \, dy$ is continuous. \ We thus have:

\begin{proposition}
Let $\mu = \nu^2$ and assume that $d \nu(t) = g(t)dt$, with $g \in L^1([0,1])$. \ Then $d\mu(t)=f(t)dt$, where 
\begin{equation}
f(a)=\int_{a}^1 g(\frac{a}{x}) g(x) \frac{1}{x} \, dx  \hspace*{.2in} (\mbox{\rm a.e.}).
\end{equation}
\end{proposition}

\begin{remark}
For future use, we note that, since $g$ vanishes outside $[0,1]$, one has
\begin{equation} \int_{a}^1 g(\frac{a}{x}) g(x) \frac{1}{x} \, dx = \int_{0}^1 g(\frac{a}{x}) g(x) \frac{1}{x} \, dx .
\end{equation}
This displays the integral as a convolution $g * g$ with respect to Haar measure on the multiplicative semigroup $(0,1)$;  we will consider this more fully in Section \ref{section5}.
\end{remark}

A first consequence of (\ref{neweq}) is a concrete expression for $f$ if $g$ is a polynomial;  the proof is merely a computation.

\begin{theorem}  \label{th:polysquare}
Suppose that $g$ is a polynomial, say $g(x) \equiv \sum_{i=0}^n a_i x^i$, positive on $[0,1]$ and inducing a probability measure $d \nu(t) = g(t) \chi_{(0,1)}(t) dt$. \ Then $\mu = \nu^2$ is absolutely continuous with respect to Lebesgue measure and with Radon-Nikodym derivative $f \cdot \chi_{(0,1)}$ where $f$ is given by
\begin{eqnarray*}
f(x) &=& \sum_{i=0}^{n-1} \left(\left(\frac{1-x^{i+1}}{i+1}\right) \sum_{j=0}^{n-i-1} a_j a_{j+i+1} x^j \right) +  \\
     && -\ln x \cdot \sum_{i=0}^{n} a_i^2 x^i \\
     &+&\sum_{i=-n-1}^{-2} \left(\left(\frac{1-x^{i+1}}{i+1}\right) \sum_{j=-i-1}^{n} a_j a_{j+i+1} x^j \right) .\\
\end{eqnarray*}
\end{theorem}

There are various consequences of this that seem (at least to us) somewhat surprising;  here, given the simple nature of $1 \, dt$, is one.

\begin{corollary}    \label{cor:1dtsqis}
Let $\nu$ be the measure on $[0,1]$ given by $d \nu(t) = 1 \,dt$; i.e., $\nu$ is Lebesgue measure on $[0,1]$. \ Then $\nu^2$ is given by $-\ln t \,dt$.
In particular, $(1 \, dt)^2$ is singular at the origin.
\end{corollary}

\begin{remark}
We also observe that the square of any polynomial measure is singular at the origin and vanishes at $1$. \ (Hindsight in comparing the areas of $p^{-1}$ applied to a small interval near $0$ and a small interval near $1$ renders this plausible.) \  As well, of course, the square root of $1 \,dt$ cannot possibly be a polynomial measure, which makes it unlikely that it arises even from a function continuous on $[0,1]$. \ Some numerical experiments using \textit{Mathematica} \cite{Wol} suggest that what is needed is some function zero at $0$ and with a singularity at $1$.
\end{remark}

We pause to record a curiosity;  the only information below not available merely from moment computations is the Berger measure associated with the square of $1 \, dt$.

\begin{corollary}
The Aluthge transform of the weighted shift associated with the Berger measure $-\ln t \, dt$ is $A_3$, the third Agler shift;  that is,
$$AT(A_2^{(2)}) = A_3.$$
\end{corollary}

\begin{proof}
Compute moments, or note that since $-\ln t \, dt = (1 \, dt)^2$, $1 \, dt = \sqrt{-\ln t \, dt}$ and use the relationship between the Aluthge transform and the square root.
\end{proof}

Note that the possible generalization of this tidy relationship to other Agler shifts $A_j$ does not hold.

Returning to the general computation, the other form of $g$ for which it is clear that the integral computation is tractable is $g$ a sum of monomials in $x^r$ with $r > -1$;  we leave the computation of the resulting $f$ to the interested reader.

\section{Enter the Laplace Transform} \label{section4}

We are indebted to Ameer Athavale for both the following result and for the method by which it is proved (\cite{At});  the proof of a more general result appears in Theorem \ref{th:pthpow1dt}.

\begin{theorem}  \label{th:squarertof1dt}
The square root measure for $1 \, dt$ (on $[0,1]$) is $\frac{1}{\sqrt{\pi}} (-\ln t)^{(-1/2)} \, dt.$
\end{theorem}

The key to this and allied results is the Laplace transform, and the movement of the moment problem from the interval $[0,1]$ to the interval $[0, \infty)$. \ Recall that the Laplace transform of a function $h$ is $H \equiv \mathcal{L}\{h\}$ where
$$
H(s) := \int_0^\infty e^{-s t} h(t) \, dt.
$$
Now suppose that $F := \mathcal{L}\{g\}$. \ Then
\begin{eqnarray*}
F(s+1) &=&  \mathcal{L}\{e^{-t} g(t)\}(s) \hspace{.2in}\mbox{\rm ``First Shifting Theorem''} \\
       &=&   \int_0^\infty e^{-s t} e^{-t} g(t) \, dt  \\
       &=&   \int_0^1 u^{s+1} \frac{1}{u} f(u) \, du  \hspace{.2in} (u := e^{-t}, f(u):= g(-\ln u) ) \\
       &=& \int_0^1 u^s f(u) \, du. \\
\end{eqnarray*}

Therefore we have the following result.

\begin{proposition}
Suppose that $\gamma(s) \equiv F(s+1)$ interpolates a sequence $(\gamma_n)$ we wish to ``match.'' \ Then the Inverse Laplace Transform, appropriately shifted to $[0,1]$, is the Radon-Nikodym derivative $f$ of the measure we seek. \ That is, $\gamma_n=\int_0^1 u^n f(u)\, du.$ 
\end{proposition}

As a consequence, we obtain the $q$-th power of $1 \, dt$ for any positive $q$, subsuming Athavale's result for the square root ($q = 1/2$) and also the earlier result for $(1 \, dt)^2$.

\begin{theorem}   \label{th:pthpow1dt}
The $q$-th power of $1 \, dt$ (on $[0,1]$)  for $q > 0$ is
$$f(u) \, du = \frac{1}{\Gamma(q)} (-\ln u)^{q-1} \, du$$
(where $\Gamma$ denotes the classical Gamma function). \ 
Alternatively, the Berger measure of the weighted shift whose moment sequence is  $\left(\sqrt[q]{\frac{1}{n+1}}\right)_{n=1}^\infty$ is \newline $\frac{1}{\Gamma(q)} (-\ln u)^{q-1} \, du.$

\end{theorem}

\begin{proof}
This follows from  the Laplace transform result that for any $p >0$,
$$ \frac{1}{s^p} \, \leftrightarrow \, \mathcal{L}\{\frac{1}{\Gamma(p)} t^{p-1}\}$$
and the process indicated above to move the result to $[0,1]$.
\end{proof}

One obtains some odd results by this route. \ 

\begin{example}
Given a constant $c>0$, is
 $$\gamma_n = \frac{e^{-c/\sqrt{(n+1)}}}{e^{-c}}$$
 the moment sequence of a subnormal shift (a ``Hausdorff moment sequence'')? \ Showing that the function $f$ given by
 $$f(x) =  \frac{e^{-c/\sqrt{(x+1)}}}{e^{-c}}$$
 is completely monotone by examining signs of the derivatives is computationally infeasible. \ It seems equally impossible to verify that the moment matrices are positive to check $k$-hyponormality for all $k$. \ But the answer is ``yes,'' because $e^{-c/\sqrt{s}}$ is the Laplace transform of
$$\frac{c}{2\sqrt{\pi t^3}} e^{-c^2/{4t}}.$$
\end{example}

Also, while every Laplace transform table entry is potentially interesting, the results need not be very friendly.
\ For example, for the square root of $A_3$ (corresponding to the measure $2(1-t) \, dt$), with the square root having moments $\frac{\sqrt{2}}{\sqrt{n+1}\sqrt{n+2}}$, the appropriate measure involves the ``modified Bessel function of the first kind of order $0$'' and is
\begin{equation}  \label{eqsqrtA3}
\sqrt{2(1-t) \, dt} = \sqrt{2}e^{(\ln u)/2}\sum_{m=0}^\infty \frac{(\ln u)/2)^{2m}}{2^{2m} (m!)^2} \, du.
\end{equation}
For the remaining Agler shifts $A_j \; (j \ge 4)$, we have found no tables that provide any help when trying to find the square roots of the Berger measures $\mu^{(j)}(t) = (j-1) (1-t)^{j-2}$.

Note that the relationship here is not between Laplace transforms and moment sequences but between Laplace transforms and \underline{functions} that interpolate moment sequences. \ This is conceptually right as shown by a result of Hausdorff-Bernstein-Widder  (see \cite{Wi}):

 \begin{theorem}
 A function is completely monotone if and only if it is the Laplace transform of a {measure}.
 \end{theorem}

Note also that if we take the fundamental relationship
$$\gamma_n = \int_0^1 t^n \, d\mu(t)$$
and turn it into
$$\gamma(s) = \int_0^1 t^s \, d\mu(t)$$
the resulting function $\gamma$ is completely monotone (if perhaps not readily computable) as seen by moving $s$ derivatives inside the integral.

The Laplace transform approach produces some measures, but seems limited to those results arising more or less directly from tables;  in particular, we have found no measure that can be ``computed'' by setting up a (novel) target completely monotone function and applying an inverse Laplace Transform.

\section{Fourier Transform Approaches} \label{section5}

As noted above the relationship
$$
f(a) =  \int_{a}^1 g(\frac{a}{x}) g(x) \frac{1}{x} \, dx = \int_{0}^1 g(\frac{a}{x}) g(x) \frac{1}{x} \, dx
$$
between the Radon-Nikodym derivatives of a measure and its square may be regarded as the convolution of $g$ with itself with respect to Haar measure on the multiplicative semi-group $(0,1)$. \ This is consistent with Theorem \ref{th:backstepmeas}; in fact, the result in Theorem \ref{th:backstepmeas} may be reinterpreted as saying that the back step extension of a shift corresponding to $g \, dt$ is subnormal if and only if $g \in \mathcal{L}^1((0,1), \mbox{\rm dHaar})$. \ As well, in this setting there is a Fourier transform approach (which includes the transform of a measure), with the Fourier transform of the convolution being the product of the Fourier transforms. \ This looks potentially useful for the Square and Square Root Problems. \ However, there are arguments against taking this point of view. \ First, harmonic analysis on semigroups is not primarily designed to be computational, and we are interested in concrete expressions for measures. \ Second, results often assume that some function is in $\mathcal{L}^1$, and many of the functions we are particularly interested in are not (for example, ``$1$''). \ Luckily, by moving the problem to $[0, \infty)$ by the change of variables $t = - \ln x$ as we did in the case of the Laplace transform, we may have something of the best of both worlds.

We must set some notation, and in particular we will revert to indicating the domains of functions explicitly, writing, for example, not ``$ 1 \, dt$'' but ``$1 \chi_{[0,1]} \, dt$''. \ We will as well have considerable use for the Heaviside function, denoted $H$, given by
$$
H(x)=\left\{
\begin{array}{rr}
1, & x \geq 0, \\
0, & x<0 ,
\end{array}%
\right. $$
and occasional use for the Signum function Sgn,
$$
\mbox{\rm Sgn}(x)=\left\{
\begin{array}{rr}
1, & x > 0, \\
0, & x = 0, \\
-1, & x<0.
\end{array}%
\right. $$
There will be two convolutions in play, and we reserve ``$*$'' for the ordinary (additive) convolution on $\mathbb{R}$ and use ``$*_I$'' for the multiplicative convolution on $(0,1)$ which appeared in Section \ref{section3}. \ With the domains $(0,1)$ and $\mathbb{R}$ both in play, we will reserve ``$t$'' for the input variable in $(0,1)$ and ``$x$'' for the input variable in $\mathbb{R}$;  the relationship will always be that of the underlying change of variables $t = e^{-x}$ or equivalently $x = - \ln t$. \ To codify this change of variable at the level of functions, for $f$ defined on $(0,\infty)$ let $L(f)$ be the function defined by $L(f)(t) = f(- \ln t)$. \ For $f$ defined on $(0,1)$, let $E(f)$ be the function given by $E(f)(x) = f(e^{-x})$. \ It is elementary to check that the changes of variables back and forth result in
$$h = f*_I g \Leftrightarrow E(h) = E(f) * E(g),\hspace{.2in} f,g,h \, \, \mbox{\rm on } (0,1),$$
and
$$h = f*g \Leftrightarrow L(h) = L(f) *_I L(g),\hspace{.2in} f,g,h \, \, \mbox{\rm on } (0,\infty).$$
It is equally simple to check that $f \cdot \chi_{(0,1)}$ is transformed to $E(f) \cdot H$ and $f \cdot H$ to $L(f) \cdot \chi_{(0,1)}$ under the appropriate changes of variable. \ The upshot of all of this is that if we wish to solve (in either direction)
$$f \chi_{(0,1)} = (g \chi_{(0,1)}) *_I (g \chi_{(0,1)})$$
it is equivalent to solve some related equation
$$f \cdot H = (g \cdot H) * (g \cdot H)$$
in the familiar setting of $\mathbb{R}$ and under the appropriate changes of variable using $L$ and $E$.

The method to achieve some such solutions will be the use of the Fourier transform. \ Recall that for a function $f$ in $L^1(\mathbb{R})$ its Fourier transform is the function $\hat{f}$ defined by
$$\hat{f}(s) = \int_{-\infty}^\infty f(x)e^{-2 \pi i x s} \, dx.$$
Recall also that the Fourier transform is routinely extended to map from the class of ``tempered distributions'' to itself. \ This class includes objects such as the Dirac delta at $x$ (denoted $\delta_x$) and ``$\frac{1}{x}$'' (where this object is a tempered distribution associated with the function $1/x$, but not equal to it). \ A useful reference for various aspects of the theory is \cite{Ch}; the most extensive printed Fourier transform tables source of which we are aware is \cite{Er}. \ It is well known that
$$\widehat{f *g} = \hat{f} \hat{g}.$$

Before continuing we  must digress at some length to point out some subtleties that are not at first apparent (or, at least, were not to us). \ We can exhibit the difficulty by attempting to check the result in Theorem \ref{th:squarertof1dt} that
$\sqrt{(1 \chi_{(0,1)} \, dt)} = \frac{1}{\sqrt{\pi}}(-\ln t)^{-1/2} \chi_{[0,1]} \, dt$ via a Fourier transform approach. \ We must change domains, take the Fourier transforms, square the right hand side, and compare, yielding
\begin{eqnarray*}
 1 \chi_{(0,1)} \, \, \mbox{\rm on $(0,1)$} &\rightarrow& H(x)\, \,  \mbox{\rm on $(0,\infty)$}  \\
                                        &{\rightarrow}& \left(\frac{1}{2}\right) \left(\delta_0(s) -\frac{1}{i \pi s}\right)   \, \,  \hspace*{.2in} \mbox{\rm (in transform space)}\\
\end{eqnarray*}
and
\begin{eqnarray*}
  \frac{1}{\sqrt{\pi}}(-\ln t)^{-1/2} \chi_{[0,1]} \, \, \mbox{\rm on $(0,1)$} &\rightarrow& \frac{1}{\sqrt{\pi}} \frac{1}{(-\ln(e^{-x}))^{(1/2)}} H(x)\, \,  \mbox{\rm on $(0,\infty)$}  \\
                                        &=& \frac{1}{\sqrt{\pi}} \frac{1}{\sqrt{x}} H(x)\, \,  \mbox{\rm on $(0,\infty)$}  \\
                                        &{\rightarrow}& \left(\frac{1}{2\sqrt{\pi} |s|^{1/2}}\right) \left(1 - i \mbox{\rm Sgn}(s)\right) \\
                                        && \, \,  \hspace*{.2in} \mbox{\rm (in transform space),} \\
\end{eqnarray*}
leaving us to compare
$$
\left(\frac{1}{2}\right) \left(\delta_0(s) -\frac{1}{i \pi s}\right)
$$
with
$$
\left(\frac{1}{2\sqrt{\pi} |s|^{1/2}}\right)^2 \left(1 - i \mbox{\rm Sgn}(s)\right)^2  
= \frac{-i \mbox{\rm Sgn}(s)}{2 \pi |s|}  + \left(\frac{1}{4 \pi |s|}\right) (1 - (\mbox{\rm Sgn}(s))^2) .  
$$
This presents us with the possible identity
$$\frac{\delta_0(s)}{2} \, = \, \left(\frac{1}{4 \pi |s|}\right) (1 - (\mbox{\rm Sgn}(s))^2),$$
and we can find no interpretation of the right hand side that renders it meaningful, even as a tempered distribution. \ (For example, one might try to approximate tempered distributions by limits of ordinary functions -- see \cite[page 127 f.f.]{Ch} -- but this does not seem to yield anything useful.)

On the other hand, it is possible to check that 
$$2t \chi_{(0,1)} \, dt = \left(\frac{\sqrt{2}}{\sqrt{\pi}}t(-\ln t)^{-1/2} \chi_{[0,1]} \, dt\right)^2$$ 
(which follows from Theorem \ref{th:squarertof1dt} and Proposition \ref{prop:restfollow}) in this way. \ (Neither we nor \textit{Mathematica} \cite{Wol} can do this in the ``square root'' approach -- by taking the square root of the Fourier transform of the left hand side and finding the Inverse Fourier transform -- but it can be done in the direction in which we square a Fourier transform and compare.)

What has gone wrong in one case but not the other?  In the first place, it is well-known that the product of two tempered distribution need not be a tempered distribution, with, as the standard example, $\delta_0^2$ not defined, and there are other anomalies (see \cite[page 402]{Ka}). \ Thus even the Square Problem takes us on to delicate ground if we depart from ordinary functions. \ A second difficulty is that if we wish to view things from the point of semi-algebras and convolutions, we must have functions in $\mathcal{L}^1$, and $1 \chi_{(0,1)} \, dt$ is not in $L^1((0,1), d\mbox{\rm Haar}) = L^1((0,1), \frac{1}{t}\, dt)$. \ Put differently, the computational result in Theorem \ref{th:polysquare} includes functions (e.g., $1 \chi_{(0,1)}$) which might not be suitable for transform methods.

What muddies matters still further, however, is the fact that sometimes things work anyway. \ Viewing $\delta_0(x)$ as the limit $\lim_{\epsilon \rightarrow 0} \frac{\epsilon}{\pi(x^2 + \epsilon^2)}$, and the tempered distribution $\frac{1}{x^2}$ as the limit $\lim_{\epsilon \rightarrow 0} \frac{x^2 -\epsilon^2}{(x^2 + \epsilon^2)}$  (cf. \cite[pages 127--130]{Ch}), one can check that $(1 \chi_{[0,1]}\, dt)^2 = -\ln t \chi_{[0,1]}\,dt$ via this approach. \ This works \emph{in spite of} the fact that, in squaring, we are, in one way or another, producing  $\delta_0^2$, which ought not even to be defined.

This state of affairs is hardly satisfactory, although surely one is on safe ground if what appears is always ordinary functions. \ In what follows we present results not sticking to this limitation but which, although discovered by Fourier transform methods, are proved by simple moment calculations once a candidate has been discovered. \ In short, at present the Fourier transform approach is heuristic at best: we use it to propose a solution, and once this is known, we must check it by comparing moments. \ Surely a more satisfactory resolution of these delicacies is desirable.

Let us finally point out to the reader wishing to experiment using \textit{Mathematica} \cite{Wol} that the program regards $\delta_0(a-b t)$ not as $\delta_{a/b}(t)$ but as $\frac{1}{b}\delta_{a/b}(t)$, and there are other anomalies that make considerable caution necessary.

The following ``recognition of a square'' was discovered by taking the Fourier transform, squaring, and taking the inverse
Fourier transform with the aid of \textit{Mathematica} \cite{Wol}.

\newpage
\begin{proposition}
Let $t_0$ be in $[0,1]$ and let $m$ be a non-negative integer. \ Then for $0 < \lambda < 1$,
\begin{eqnarray*}
(\lambda \delta_{t_0} &+& (1-\lambda) (m+1)t^m  \chi_{(0,1)}(t) \, dt)^2 \\
&=& \lambda^2 \delta_{2 t_0} + (1-\lambda)^2 (m+1)^2 (- \ln t) t^m  \chi_{(0,1)}(t)\, dt \\
&+& 2 \lambda (1-\lambda) (m+1)\frac{t^m}{t_0^{m+1}}  \chi_{(0,t_0)}(t)\, dt.
\end{eqnarray*}
\end{proposition}

\begin{proof}
Compute moments.
\end{proof}

Observe that the above is true both for $t_0 \in (0,1)$ (for which convolution/transform methods are appropriate since this is a measure on $(0,1)$) and for $t_0 = 1$ or $t_0 = 0$ for which such methods are not apparently appropriate.

Noting that after the change of variables we will always be considering functions living on $[0, \infty)$ (sometimes known as ``causal'' functions), it is natural to look at functions of the form $f(x)H(x)$ for which Fourier transforms are known. \ One such is $x^n e^{-\alpha x} H(x)$ with $\alpha > 0$ and $n$ a non-negative integer. \ These correspond to functions $(-\ln t)^n t^\alpha \chi_{(0,1)}$ on $(0,1)$, and it turns out that one can handle linear combinations of these and $\delta_{t_j}$'s. \ We content ourselves, however, with two examples, since certain cross terms in transform space require, for their Inverse Fourier Transform, partial fraction decompositions that rapidly become unwieldy.

\begin{example}
Consider the measures  $\nu_1$ and $\nu_2$ given by
$$\nu_1(t) = \frac{196}{169} (-\ln t)t^{\frac{1}{13}} \chi_{(0,1)}(t) \, dt$$
and
$$\nu_1(t) = 256 (-\ln t)^2 t^7 \chi_{(0,1)}(t) \, dt.$$
Then
\begin{eqnarray*}
(\frac{1}{3} \nu_1 + \frac{2}{3} \nu_2)^2 &=&  \frac{19208}{771147} t^{1/13} (-\ln t)^3 \chi_{(0,1)}(t) \, dt +  \frac{131072}{135} t^7 (- \ln t)^5 \chi_{(0,1)}(t) \, dt  \\
&&+ \frac{4239872}{12301875}-\frac{4239872 t^{1/13}}{12301875} + \frac{652288 t^{1/13} (-\ln t)}{820125}\\ &+&\frac{1304576 t^7 (-\ln t)}{820125} +\frac{50176 t^7 (-\ln t)^2}{18225}.  \\
\end{eqnarray*}

\noindent Changing domains and taking the Fourier Transform yields
\begin{eqnarray*}
 \frac{196}{3*169} (-\ln t)t^{1/13} \chi_{(0,1)}(t) + \frac{2}{3}*256 (-\ln t)^2 t^7 \chi_{(0,1)}(t) \,\,\, \, \mbox{\rm on $(0,1)$} &\rightarrow&  \\
 \frac{196}{3*169}x e^{-x/13} H(x) + \frac{2*256}{3}x^2 e^{-7x} H(x) \, \,\,\,  \mbox{\rm on $(0,\infty)$}
                                        &{\rightarrow}& \\
                                        \frac{4}{3} \left(-\frac{49}{(i-26 \pi  s)^2}+\frac{256 i}{(-7 i+2 \pi  s)^3}\right) \hspace*{.2in} \mbox{\rm (in transform space),}
\end{eqnarray*}
which, upon squaring, gives
\begin{equation}  \label{eq:FTspace}
\frac{38416}{9 (i-26 \pi  s)^4}-\frac{1048576}{9 (-7 i+2 \pi  s)^6}-\frac{401408 i}{9 (i-26 \pi  s)^2 (-7 i+2 \pi  s)^3}.
\end{equation}
Taking the Inverse Fourier Transform yields
\begin{eqnarray}  \label{eq:origspace}
&&\left[\frac{19208 e^{-x/13} x^3}{771147} H(x)\right] + \nonumber  \\
&& + \left[\frac{131072}{135} e^{-7 x} x^5 H(x)\right] +   \\
&& + \left[\left(\frac{4239872 e^{-7 x}}{12301875}-\frac{4239872 e^{-x/13}}{12301875}+\frac{1304576 e^{-7 x} x }{820125}\right)H(x)\right. +  \nonumber\\
&&\left.\left(\frac{652288 e^{-x/13} x }{820125}+\frac{50176 e^{-7 x} x^2 }{18225}\right) H(x)\right],  \nonumber  \\
 \nonumber \end{eqnarray}
where terms in (\ref{eq:FTspace}) correspond with terms in square brackets in (\ref{eq:origspace}). \ (As noted, the cross term in the third term of (\ref{eq:FTspace}) spawns terms because of a partial fractions expansion.) \ Transforming back to the interval $(0,1)$ yields the result, which may be checked directly by evaluating moments.
\end{example}

\begin{example}
Consider the measures  $\nu_1$ and $\nu_2$ given by
$$\nu_1(t) =\delta_{e^{-1/5}} \, dt$$
and
$$\nu_1(t) = \frac{8}{3} (-\ln t)^3 t^2 \chi_{(0,1)}(t) \, dt.$$
Then
\begin{eqnarray*}
(\frac{1}{3} \nu_1 + \frac{2}{3} \nu_2)^2 &=&  \frac{1}{9} \delta_{e^{-2/5}} + \frac{64}{2835}(-\ln t)^7 t^2 \chi_{(0,1)}(t) \, dt + \\
                             && + \frac{32}{27} t^2 (-\ln t - 1/5)^3 \chi_{(0,1)}(t) \, dt . \\
\end{eqnarray*}

This is a computation, with the only possible difficulty arising from the Inverse Fourier Transform of $\frac{4 e^{-\frac{2}{5} i \pi  s}}{9 (-i+\pi  s)^4}$, the cross term after changing domains, taking the Fourier Transform, and squaring. \ But with the aide of \textit{Mathematica} \cite{Wol} one finds the Inverse Fourier Transform is just $\frac{32 e^{-2 \left(x-\frac{1}{5 }\right) } \left(x-\frac{1}{5 }\right)^3 }{27} H(x)$, which gives rise to the final term in the statement upon returning to the domain $(0,1)$.
\end{example}

\section{Closing Remarks} \label{section6}

\medskip

We conclude with some remarks and questions. \ First, as might be expected, the Fourier Transform approach as applied to purely atomic measures does not yield anything new;  what results is a slightly disguised version of what has already been done. \ Second, the attentive reader will note that there is no application of Fourier Transforms to yield a new square root (leaving out trivialities such as ``discovering'' the square root of a known square). \ We know of no such example -- the difficulty seems to be that computing the Inverse Fourier Transform of the square root of something in transform space is usually computationally intractable. \ We point out the anomaly that the Laplace transform methods do yield new square roots, but in a way that makes no use of the (known) measure but only of the moments. \ It would be very useful to find an efficacious method of using the known measure to aid in finding the square root measure, but (\ref{eqsqrtA3}) indicates that this may be difficult.

If there is a larger point to be taken away from this investigation, it is perhaps greater awareness that the family of probability measures with compact support is a large collection of measures including members very far from the friendly few which have usually been the objects of study, at least in the context of Berger measures of subnormal weighted shifts.

\medskip

\noindent \textbf{Acknowledgment}

\medskip

This work has been pursued with visits both to Bucknell University and the University of Iowa, and the authors wish to express their gratitude for the warm hospitality from the Mathematics Departments and institutions during their visits.

\end{document}